\documentclass[twoside]{amsart}
\usepackage{amsmath,amsfonts,amsthm,mathrsfs,amscd}
\usepackage[numeric,initials]{amsrefs}
\usepackage{amssymb}
\usepackage[PostScript=dvips,dpi=1270,heads=LaTeX]{diagrams}  %curlyvee
\diagramstyle{small,thin,labelstyle=\scriptstyle}

\newcommand{\fust}[1][f]{{#1}^{\ast}}
\newcommand{\flst}[1][f]{{#1}_{\ast}}
\newcommand{\finv}[1][f]{{#1}^{-1}}

\newcommand{\tensor}{\otimes}

\newcommand{\ZZ}{\mathbb{Z}}
\newcommand{\QQ}{\mathbb{Q}}

\newtheorem{thm}{Theorem}%[section]
\newtheorem{lem}[thm]{Lemma}
\newtheorem{prop}[thm]{Proposition}
\newtheorem*{lem*}{Lemma}

\newtheorem*{thm*}{Theorem}
\newtheorem*{prop*}{Proposition}

\theoremstyle{definition}

\theoremstyle{remark}

\newtheorem*{ex*}{Example}

\theoremstyle{plain}

\newcommand{\define}[1]{\emph{#1}}

\newcommand{\shf}[1]{\mathscr{#1}}
\newcommand{\OX}{\shf{O}_X}

\def\overbar#1#2#3{{%
	\setbox0=\hbox{\displaystyle{#1}}%
	\dimen0=\wd0
	\advance\dimen0 by -#2 
	\vbox {\nointerlineskip \moveright #3 \vbox{\hrule height 0.3pt width \dimen0}%
		\nointerlineskip \vskip 1.5pt \box0}%
}}

\newcommand{\PP}{P}
\newcommand{\famX}{\mathfrak{X}}
\newcommand{\famXsing}{\mathfrak{X}_{\mathit{sing}}}
\newcommand{\dualX}{X^{\vee}}
\newcommand{\etaW}{\eta_W}

\newcommand{\singplus}{\!\!\! \bigoplus_{t \in C \setminus C_0} \!\!\!}
\newcommand{\class}[1]{\lbrack #1 \rbrack}

\DeclareMathOperator{\codim}{codim}

\begin{document}

%========================================================
\title[Two observations]{Two observations about normal functions}
\author[C.~Schnell]{Christian Schnell}
\address{Department of Mathematics, Statistics \& Computer Science \\
University of Illinois at Chicago \\
851 South Morgan Street \\
Chicago, IL 60607}
\email{cschnell@math.uic.edu}
%\date{\today}
\begin{abstract}
Two simple observations are made: (1) If the normal function associated to
a Hodge class has a zero locus of positive dimension, then it has a singularity. (2)
The intersection cohomology of the dual variety contains the cohomology of the
original variety, if the degree of the embedding is large.
\end{abstract}
\maketitle
%========================================================

This brief note contains two elementary observations about normal functions and their
singularities that arose from a conversation with Pearlstein. The proofs are very
simple, and it is quite possible that both statements are known; nevertheless, it
seemed useful to me to have them written down. Throughout, $X$ will be a smooth
projective variety of dimension $2n$, and $\eta$ a primitive Hodge class of weight
$2n$ on $X$, say with integer coefficients. We shall also let $\pi \colon \famX \to
\PP$ be the universal hypersurface section of $X$ by some very ample class $H$, with
discriminant locus $\dualX \subset \PP$. 

\section{The zero locus of a normal function}

Here we show that if the zero locus of the normal function associated to a Hodge
class $\eta$ contains an \emph{algebraic} curve, then the normal function must be
singular at one of the points of intersection between $\dualX$ and the closure of
the curve. 

\begin{prop} \label{prop:1}
Let $\nu$ be the normal function on $\PP \setminus \dualX$, associated to
a non-torsion Hodge class $\eta \in H^{2n}(X, \ZZ) \cap H^{n,n}(X)$. Assume that the zero locus
of $\nu$ contains an algebraic curve, and that the divisor $H$ is sufficiently ample.
Then $\nu$ is singular at one of the points where the closure of the curve
meets $\dualX$.
\end{prop}

Before giving the proof, we briefly recall some definitions. In general, a normal
function for a variation of Hodge structure of odd weight on a complex manifold $Y_0$
has an associated \define{cohomology class}. If $H_{\ZZ}$ is the local system underlying the
variation, then a normal function $\nu$ determines an extension of local systems
\begin{equation} \label{eq:ext}
\begin{diagram}
0 &\rTo& H_{\ZZ} &\rTo& H'_{\ZZ} &\rTo& \ZZ &\rTo& 0.
\end{diagram}
\end{equation}
The cohomology class $\class{\nu} \in H^1 \bigl( Y_0, H_{\ZZ} \bigr)$ of the normal
function is the image of $1 \in H^0 \bigl( Y_0, \ZZ \bigr)$ under the connecting
homomorphism for the extension.

In particular, the normal function $\nu$ associated to a Hodge class $\eta$
determines a cohomology class $\class{\eta} \in H^1 \bigl( \PP \setminus \dualX, R^{2n-1}
\flst[\pi] \ZZ \bigr)$. With rational coefficients, that class 
can also be obtained directly from $\eta$ through the Leray spectral sequence
\[
	E_2^{p,q} = H^p \bigl( \PP \setminus \dualX, R^q \flst[\pi] \QQ \bigr)
		\Longrightarrow H^{p+q} \bigl( P \times X \setminus \pi^{-1}(\dualX), \QQ).
\]
To wit, the pullback of $\eta$ to $P \times X \setminus \pi^{-1}(\dualX)$ goes to
zero in $E_2^{0, 2n}$ because $\eta$ is primitive, and thus gives an element of
$E_2^{1, 2n-1}$; this element is precisely $\class{\eta}$.

\begin{lem} \label{lem:W0}
Let $C_0 \to P \setminus \dualX$ be a smooth affine curve mapping into the zero locus of $\nu$,
and let $W_0$ be the pullback of the family $\famX$ to $C_0$. Then the image of the Hodge class
$\eta$ in $H^{2n}(W_0, \QQ)$ is zero.
\end{lem}
\begin{proof}
Let $\psi \colon W_0 \to C_0$ be the obvious map. The pullback of $R^{2n-1}
\flst[\pi] \QQ$ to $C_0$ is naturally isomorphic to $R^{2n-1} \flst[\psi] \QQ$;
moreover, when $\nu$ is restricted to $C_0$, its class is simply the image of
$\class{\eta}$ in $H^1 \bigl( C_0, R^{2n-1} \flst[\psi] \QQ \bigr)$. That image has
to be zero, because $C_0$ maps into the zero locus of $\nu$.

Now let $\eta_0 \in H^{2n} \bigl( W_0, \QQ \bigr)$ be the image of the Hodge class
$\eta$. The Leray spectral sequence for the map $\psi$ gives a short exact sequence
\begin{diagram}[size=0pt]
	0 & \rTo & H^1 \bigl( C_0, R^{2n-1} \flst[\psi] \QQ \bigr) & \rTo & 
		H^{2n} \bigl( W_0, \QQ \bigr) 
		& \rTo & H^0 \bigl( C_0, R^{2n} \flst[\psi] \QQ \bigr) & \rTo & 0,
\end{diagram}
and as before, $\eta_0$ actually lies in $H^1 \bigl( C_0, R^{2n-1} \flst[\psi] \QQ
\bigr)$. Because the spectral sequences for $\psi$ and $\pi$ are compatible,
$\eta_0$ is equal to the image of $\class{\eta}$; but we have already seen that 
this is zero.
\end{proof}

Returning to our review of general definitions, let $\nu$ be a normal function on $Y_0$. 
When $Y_0 \subseteq Y$ is an open subset of a bigger complex manifold, one can look
at the behavior of $\nu$ near points of $Y \setminus Y_0$. The
\define{singularity} of $\nu$ at a point $y \in Y \setminus Y_0$ is by definition the
image of $\class{\nu}$ in the group
\[
	\lim_{U \ni y} H^1 \bigl( U \cap Y_0, H_{\ZZ} \bigr),
\]
the limit being over all analytic open neighborhoods of the point \cite{BFNP}*{p.~1}.
If the singularity is non-torsion, $\nu$ is said to be \define{singular} at the point $y$.

When $\nu$ is the normal function associated to a non-torsion primitive Hodge class $\eta \in
H^{2n}(X, \ZZ)$, Brosnan, Fang, Nie, and Pearlstein \cite{BFNP}*{Theorem~1.3}, and
independently de Cataldo and Migliorini \cite{dCM}*{Proposition~3.7}, have
proved the following result: Provided the vanishing cohomology of the smooth fibers
of $\pi$ is nontrivial, $\nu$ is singular at a point $p \in \dualX$ if, and
only if, the image of $\eta$ in $H^{2n} \bigl( \pi^{-1}(p), \QQ \bigr)$ is nonzero.
By recent work of Dimca and Saito \cite{DS}*{Theorem~6}, it suffices to take $H = d
A$, with $A$ very ample and $d \geq 3$.

\begin{proof}[Proof of Proposition~\ref{prop:1}]
Let $C$ be the normalization of the closure of the curve in the zero locus. Pulling
back the universal family $\famX \to \PP$ to $C$ and resolving singularities, we
obtain a smooth projective $2n$-fold $W$, together with the following two maps:
\begin{diagram}
	W & \rTo^{\lambda} & X \\
	\dTo^{\psi} \\
	C 
\end{diagram}
This may be done in such a way that the general fiber of $\psi$ is a smooth
hypersurface section of $X$; let $C_0 \subseteq C$ be the open subset where this
holds, and $W_0 = \finv[\psi](C_0)$ its preimage. Assume in addition that, for each
$t \in C \setminus C_0$, the fiber $E_t$ is a divisor with simple normal crossing
support. The map $\lambda$ is generically finite, and we let $d$ be its degree.

Let $\etaW = \fust[\lambda](\eta)$ be the pullback of the Hodge class to $W$. By
Lemma~\ref{lem:W0}, the restriction of $\etaW$ to $W_0$ is zero. Consider now the
exact sequence
\begin{diagram}[width=3em]
	H^{2n}(W, W_0, \QQ) & \rTo^i &  H^{2n}(W, \QQ) & \rTo &
			H^{2n}(W_0, \QQ).
\end{diagram}
By what we have just observed, $\etaW$ belongs to the image of the map $i$, say $\etaW =
i(\alpha)$. Under the nondegenerate pairing (given by Poincar\'e duality)
\[
	H^{2n}(W, W_0, \QQ) \tensor \singplus H^{2n}(E_t, \QQ) \to \QQ,
\]
and the intersection pairing on $W$, the map $i$ is dual to the restriction map
\[
	H^{2n}(W, \QQ) \to \singplus H^{2n}(E_t \QQ),
\]
and so we get that
\[
	d \cdot \int_X \eta \cup \eta = 
	\int_W \etaW \cup \etaW = \bigl\langle i(\alpha), \etaW \bigr\rangle = 
	\sum_{t \in C \setminus C_0} \bigl\langle \alpha, i_t^{\ast}(\etaW) \bigr\rangle
\]
where $i_t \colon E_t \to W$ is the inclusion. But the first integral is nonzero,
because the intersection pairing on $X$ is definite on the subspace of primitive
$(n,n)$-classes. We conclude that the pullback of $\eta$ to at least one
of the $E_t$ has to be nonzero. 

By construction, $E_t$ maps into one of the singular fibers of $\pi$, say to
$\pi^{-1}(p)$, where $p$ belongs to the intersection of $\dualX$ with the closure of
the curve. Thus $\eta$ has nonzero image in $H^{2n} \bigl( \pi^{-1}(p), \QQ \bigr)$;
by the result of \cites{BFNP,dCM} mentioned above, $\nu$ has to be singular at $p$,
concluding the proof.
\end{proof}

I do not know whether a ``converse'' to Proposition~\ref{prop:1} is true; that is to say,
whether the normal function associated to an algebraic cycle on $X$ has to have a
zero locus of positive dimension for sufficiently ample $H$. If it was, this would 
give one more equivalent formulation of the Hodge conjecture.

\section{Cohomology of the discriminant locus}

Pearlstein pointed out that the singularities of the discriminant locus should be
complicated enough to capture all the primitive cohomology of the original variety,
once $H = dA$ is a sufficiently big multiple of a very ample class.  In this section,
we give an elementary proof of this fact for $d \geq 3$. 

To do this, we need a simple lemma, used to estimate the codimension of loci in
$\dualX$ where the fibers of $\pi$ have a singular set of positive dimension. Let
\[
	V_d = H^0\bigl(X, \OX(dA)\bigr)
\]
be the space of sections of $dA$, for $A$ very ample. 

\begin{lem} \label{lem:codim}
Let $Z \subseteq X$ be a closed subvariety of positive dimension $k > 0$. Write
$V_d(Z)$ for the subspace of sections that vanish along $Z$. Then
\[
	\codim \bigl( V_d(Z), V_d \bigr) \geq \binom{d+k}{k}.
\]
\end{lem}

\begin{proof}
Since $A$ is very ample, we may find $(k+1)$ points $P_0, P_1, \dotsc, P_k$ on $Z$,
together with $(k+1)$ sections $s_0, s_1, \dotsc, s_k \in V_1$, such that each $s_i$
vanishes at all points $P_j$ with $j \neq i$, but does not vanish at $P_i$. Then all
the sections
\[
	s_0^{\tensor i_0} \tensor s_1^{\tensor i_1} \tensor \dotsm \tensor
		s_k^{\tensor i_k} \in V_d,
\]
for $i_0 + i_1 + \dotsb + i_k = d$, are easily seen to be linearly independent on
$Z$. The lower bound on the codimension follows immediately.
\end{proof}

We now use the estimate to make Pearlstein's suggestion precise. As one further bit 
of notation, let $\famXsing \subseteq \famX$ stand for the union of all the singular
points in the fibers of $\pi$. It is well-known that $\famXsing$ is a projective
space bundle over $X$, and in particular smooth.

\begin{prop}
Let $H = d A$ for a very ample class $A$. If $d \geq 3$, then the map
$\phi \colon \famXsing \to \dualX$ is a small resolution of singularities, and
therefore
\[
	\mathit{IH}^{\ast}(\dualX, \QQ) \simeq H^{\ast}(\famXsing, \QQ).
\]
In particular, $H^{\ast}(X, \QQ)$ is a direct summand of $\mathit{IH}^{\ast}(\dualX, \QQ)$
once $d \geq 3$.
\end{prop}

\begin{proof}
By \cite{DS}*{p.~Theorem~6}, the discriminant locus is a divisor in $P$ once $d \geq
3$. This means that there are hyperplane sections of $X$ with exactly one ordinary
double point \cite{Voisin}*{p.~317}. The map $\phi$ is then birational, and therefore
a resolution of singularities of $\dualX$.
To prove that it is a small resolution, take a Whitney stratification of $\PP$ in
which $\dualX$ is a union of strata, and such that over each stratum, the maps $\pi$
and $\phi$ are topologically trivial.  The dimension of the singular set of the
fibers of $\pi$ is then constant along each stratum. 

Let $S \subseteq \dualX$ be an arbitrary stratum along which the singular set of the
fiber has dimension $k > 0$. At a general point $t \in S$, there then has to be an
irreducible component $Z$ in the singular locus of $\finv[\pi](t)$ that remains
singular to first order along $S$. Now a tangent vector to $S$ may be represented by
a section $s$ of $dA$, and the condition that $Z$ remain singular to first order is
that $s$ should vanish along $Z$. By Lemma~\ref{lem:codim}, the space of such
sections has codimension at least $\binom{d+k}{k}$, and a moment's thought shows that,
therefore,
\[
	\codim(S, \dualX) \geq \binom{d+k}{k} - 1.
\]

This quantity is evidently a lower bound for the codimension of the locus where the
fibers of $\phi$ have dimension $k$. In order for $\phi$
to be a small resolution, it is thus sufficient that
\[
	\binom{d+k}{k} - 1 > 2 k
\]
for all $k > 0$. Now one easily sees that this condition is satisfied provided
that $d \geq 3$. This proves the first assertion; the second one is a general fact
about intersection cohomology. Finally, $H^{\ast}(X, \QQ)$ is a direct summand of
$H^{\ast}(\famXsing, \QQ)$ because $\famXsing$ is a projective space bundle over $X$,
and the third assertion follows.
\end{proof}

The proof shows that, as in the theorem by Dimca and Saito, $d \geq 2$ is sufficient
in most cases. A related result, showing the effect of taking $H$ sufficiently ample,
was obtained by Clemens; he noticed that, as a consequence of Nori's connectivity
theorem, one has an isomorphism
\[
	H^{2n}(X, \QQ)_{\mathrm{prim}} \simeq 
		H^1 \bigl( P \setminus \famX, (R^{2n-1} \flst[\pi] \QQ)_{\mathrm{van}} \bigr),	
\]
once $H$ is sufficiently ample \cite{Schnell}*{Proposition~7 on p.~11}.

%++++++++++++++++++++++++++++++++++

\end{document}